\documentclass[12pt,a4paper]{article}

\usepackage[whole]{bxcjkjatype}

\usepackage{amsmath}
\usepackage{amssymb}
\usepackage{bussproofs}
\usepackage{tikz}
\usepackage{stmaryrd}
\usepackage{cmll}
\usepackage{mathptmx}
\usepackage{xcolor}
\usepackage{amsthm}
\usepackage{framed}
\usepackage{here}
\usepackage{comment}
\usepackage{seqsplit}

\newcommand{\tensor}{\mbox{$\otimes$}}
\newcommand{\pa}{\bindnasrepma}
\newcommand*{\oldneg}{\mathord{\sim}}

\newtheorem{defn}{Definition}
\newtheorem{lem}{Lemma}
\newtheorem{thm}{Theorem}
\newtheorem{prop}{Proposition}

\newtheorem{sublem}{Sublemma}
\newtheorem{fact1}{Fact}
\newtheorem{rem}{Remark}

\fontsize{14pt}{0pt}\selectfont
\title{A Note on Switching Conditions for the Generalized Multiplicative Connectives}
\fontsize{12pt}{0pt}\selectfont
\author{Yuki Nishimuta\thanks{Keio University, E-mail: nishimuta@abelard.flet.keio.ac.jp} \and Mitsuhiro Okada \thanks{Keio University, E-mail: E-mail: mitsu@abelard.flet.keio.ac.jp}}
\fontsize{12pt}{0pt}\selectfont
\date{}
					
\begin{document}
\maketitle

\begin{abstract}
Danos and Regnier (1989) introduced the par-switching condition for multiplicative proof-structures and  simplified the sequentialization theorem of Girard (1987) by means of par-switching. Danos and Regnier (1989) also generalized the par-switching to a switching for $n$-ary connectives (hereafter called an $n$-ary switching) and showed that the ``expansion" property holds, namely that  any ``excluded-middle" formula admits a correct proof-net in the sense of their $n$-ary switching. They added a remark that the sequentialization theorem does not hold with their switching. Their definition of switching for $n$-ary connectives is a natural generalization of the original switching for the binary connectives. However, there are many other possible definitions of switching for $n$-ary connectives.  We give an alternative and ``natural" definition of $n$-ary switching, and we show that the proof of sequentialization theorem by Olivier Laurent with the par-switching works for our $n$-ary switching;  Consequently,  the sequentialization theorem holds for our $n$-ary switching. On the other hand, we remark that the ``expansion" property no longer holds under our switching anymore. We point out that no definition of $n$-ary switching satisfies both the sequentialization theorem and the ``expansion" property at the same time except for the purely tensor-based (or purely par-based) connectives.

\end{abstract}

\section{Introduction}
\label{intro}

The sequentialization theorem of Girard (1987) for the Multiplicative fragment of Linear Logic ($\mathsf{MLL}$) says that a (graphically represented)  multiplicative proof structure 
can be translated to a sequential $\mathsf{MLL}$-proof if his ``long-trip" condition is satisfied.  Danos and Regnier (1989) simplified the sequentialization theorem by introducing the notion of par-switching condition, which says that  a resulting  proof-structure from choosing one of two nodes for all par-links becomes a connected and acyclic graph.
Their sequentialization theorem asserts that if a given proof structure satisfies the par-switching condition, the proof-structure can be transformed to a sequential $\mathsf{MLL}$-proof. In the same paper, they generalized the usual binary multiplicative connectives, par and tensor, to $n$-ary connectives. They gave the way to introduce a pair of dual $n$-ary connectives by the notion of ``orthogonality of meeting graphs", which guarantees the main-cut elimination process of the cut between the introduction rules of the dual $n$-ary connectives. In addition, they chose one natural way of generalizing par-switching to a switching for $n$-ary connectives and they showed that (1)  the ``expansion property" holds if the switching condition is satisfied, that is, if  $\mathcal{C}(A_1,\dots,A_n), \mathcal{C}^{\ast}(\oldneg A_1,\dots,\oldneg A_n)$ has a correct proof-net from atomic-links, where   $A_i$ is an atom, and $\oldneg A_i$ is its dual in the sense of  the binary-dual, namely the usual negation, while $\mathcal{C}$ and $\mathcal{C}^{\ast}$ are new $n$-ary dual connectives, and that (2) the sequentialization theorem does not hold anymore with this switching definition. The purpose of this note is to point out that there exists another natural way of generalizing switching condition for $n$-ary connectives, and the switching condition with this alternative notion of switching leads to the sequentialization theorem at the expense of the ``expansion" property.
\par The principal formulas (of the upper-sequents) for the usual binary connectives, tensor rule and par-rule, in the typical (classical) one-sided sequent calculus can be expressed as $\vdash A_1 \  \vdash A_2$ for introduction rule of tensor $A_1\tensor A_2$, while $\vdash A_1, A_2$ for that of $A_1\pa A_2$, except for arbitrary auxiliary context formulas in $\mathsf{MLL}$. The situation can be expressed as the possible partition classes $\{(1)(2)\}$ for the tensor rule and $\{(1,2)\}$ for the par rule, partition of $\{1,2\}$.The tensor has two partition classes of singleton, while par a single class of two elements. This suggests that any $n$-ary connective rules can be introduced with a set of partition classes; for example, if $n=4$, one could consider a partition classes $\{(1,2)(3,4)\}$. In the case of the binary par-switching, Danos-Regnier use the selection function to choose one element from par-link $\{(1,2)\}$. When Danos-Regnier  generalize this switching to a $4$-ary link or rule such as $\{(1,2)(3,4)\}$, they keep this idea of the selection function and first chooses one class, either $(1,2)$ or $(3,4)$ to define the switching. There is also an alternative and natural way to define a switching; one chooses one out of each classes, namely, one out of $(1,2)$ and one out of $(3,4)$. We consider this alternative definition of switching  and remark that the switching condition with such the definition implies the  sequentialization theorem. We also remark that no ways of defining switching would satisfy both the sequentialization and expansion at the same time, except for the essentially  original binary connectives.

Danos and Regnier's par-switching condition is given via an association to Girard's long-trip condition. Several direct proofs are known that derive the sequentialization theorem from the Danos-Regnier par-switching condition \cite{DR}. To the best of our knowledge, the most well-known is Girard's proof \cite{Gir proofnet}.  Olivier Laurent gave a simple and direct proof, in his unpublished note ``Sequentialization of multiplicative proof nets" at 2013 (available at: http://perso.ens-lyon.fr/olivier.laurent \newline seqmill.pdf)  [Accessed 1 April 2018]. %
%
The main purpose of our note is to show that Laurent's proof of sequentialization  for the usual binary $\mathsf{MLL}$ connectives works for the generalized $n$-ary connectives if we take our alternative choice of the definition of switching, instead of  the definition of switching which was chosen by Danos-Regnier (1989). In the course of following Laurent's sequentialization proof, we slightly simplify some part of splitting lemma proof. The generalized connectives have been studied in Maieli (2019) and Jean-Yves Girard's paper ``transcendental syntax II: non deterministic case" at 2017, available at: http://girard.perso.math.cnrs.fr/trsy2.pdf  [Accessed 1 April 2018], where the author remarks, among others, importance of generalized connectives. However, the issue on the switching condition which we consider in this paper are not discussed in these papers.


\section{Preliminaries}

 We recall background knowledge of  generalized multiplicative connectives introduced by Danos and Regnier \cite{DR}. This material is contained in section 2 and 3 of \cite{DR}.  Background material on Multiplicative Linear Logic and proof-nets is collected in  the Appendix.

\medskip

Danos and Regnier introduced the generalized multiplicative rule satisfying the two properties of multiplicative connectives \cite[p. 188]{DR}.  A generalized connective is defined by particular instances of  generalized multiplicative rules.

\bigskip

The introduction rules of multiplicative connectives have the following two properties \cite{DR};
\begin{enumerate}
\item  all maximal subformulas of the conclusion formula  occur in  the premises,
\item  an introduction of a multiplicative connective does not require information about the contexts of the premises.
\end{enumerate}

The following is the general form of an introduction rule for a multiplicative connective. This form  satisfies the above two properties.

\begin{center} 
\def\fCenter{\ \vdash\ }
\Axiom$\fCenter \Gamma_1,  A_{11}, \dots, A_{1i_1} \ \ \cdots$
\Axiom$\fCenter\Gamma_m, A_{m1},\dots,A_{mi_m}$
\BinaryInf$\fCenter\Gamma_1,\dots, \Gamma_m, \mathcal{C}(A_1,\dots,A_n)$
\DisplayProof
\end{center}

where $ji_j\in\{1,\dots,n\}$ and $\mathcal{C}$ is an unspecified connective.

\medskip

We  consider the set $\mathcal{F}$ of partitions of a natural number  $n$. We assume that a partition $p$ has the form $p=\{(p_{11},\dots,p_{1j}),\dots(p_{k1},\dots,p_{km_{j}})\}$, $p_{ij}\in\{1,\dots, n \}$, $p\in\mathcal{F}$.  For each class $i$, a class $(p_{i1},\dots,p_{im_i})$ corresponds to a sequent $\vdash A_{i1}, \dots,\newline A_{im_{i}}$. Hence, each introduction rule of a connective $\mathcal{C}$ corresponds to a partition $p$.

\medskip
A set of introduction rules of a generalized connective $\mathcal{C}$ correspond to a partition set $P$. 
$P_{\mathcal{C}}$ denotes the partition set that corresponds to a generalized connective $\mathcal{C}$. We denote by $P_{\mathcal{C}}$ the partition set corresponding to $\mathcal{C}$ and regard it as the set of ``right rules" of $\mathcal{C}$. We consider the right rules of the generalized connective $\mathcal{C}^{\ast}$ as the left rules of $\mathcal{C}$, where $\mathcal{C}^{\ast}$ is the dual connective, in the sense of Danos-Regnier's orthogonality, which we shall explain below, of $\mathcal{C}$ (e.g. if $\mathcal{C}=\tensor$ then $\mathcal{C}^{\ast}=\pa$).
We put $\oldneg(\mathcal{C}(A_1,\dots,A_n))=\mathcal{C}^{\ast}(\oldneg A_1,\dots,\oldneg A_n)$ for some $\mathcal{C}^{\ast}$. It is required that the dual connective $\mathcal{C}^{\ast}$ of $\mathcal{C}$ is defined  such that the main step of the cut elimination proof holds for a pair $(\mathcal{C}, \mathcal{C}^{\ast})$. Thus, the duality of the generalized connectives is based on cut-elimination, and not based on logical duality.   Danos and Regnier invented  theory of meeting graph so as to define $\mathcal{C}^{\ast}$ such that the cut-elimination holds. \cite{DR}.


\begin{defn}
For any two partitions $p,q\in\mathcal{F}$, a meeting graph  $\mathcal{G}(p,q)$ is a labelled graph $(V_{1},V_{2}, E)$ as follows;
\begin{itemize}
\item $V_{1}$ is the set of upper nodes:  these nodes have labels such that the numbers in the same class of $p$ correspond to the same node.
\item $V_{2}$ is the set of lower nodes:   these nodes have labels such that the numbers in the same class of $q$ correspond to the same node.
\item$E$ is the set of edges connecting a node of $V_{1}$ and a node of $V_{2}$ such that for each number there is exactly one edge between the corresponding nodes of $p$ and $q$
\end{itemize}
\end{defn}


\begin{defn}
\cite{DR} Two partitions $p, q\in\mathcal{F}$ are orthogonal if the meeting graph $\mathcal{G}(p,q)$ is connected and acyclic. We denote $p\perp q$ if $p$ and $q$ are orthogonal.
\end{defn}
Hence, $p\perp q$ if and only if $\mathcal{G}(p,q)$ is connected and acyclic.

\begin{defn}
\cite{DR} Partition sets $P, Q \subseteq \mathcal{F}$ are orthogonal if for any $p\in P$ and $q\in Q$, $p\perp q$ holds. We denote $P\perp Q$ if $P$ and $Q$ are orthogonal.  $P^{\bot}$ is the maximal set which is orthogonal to $P$. $P^{\bot}=\{q|\forall p\in P. p\perp q\}$.

\end{defn}

For further detailed property of a partition set, see \cite{Maieli1}, although we do not need the further properties in this paper.


\begin{defn}
\cite{DR} A pair of n-ary generalized connectives ($\mathcal{C}$, $\mathcal{C}^{\ast}$) is a pair of  non-empty finite partition sets of a natural number $n$ ($P_{\mathcal{C}}, P_{\mathcal{C}^{\ast}}$) such that $(P_{\mathcal{C}})^{\bot}=P_{\mathcal{C}^{\ast}}$ and  $(P_{\mathcal{C}^{\ast}})^{\bot}=P_{\mathcal{C}}$.
\end{defn}

Note that $\mathcal{C}^{\ast}$ is uniquely determined by $\mathcal{C}$ because $(P_{\mathcal{C}})^{\bot}$ is the maximal set for $P_{\mathcal{C}}$. This fact means that  left rules of  $\mathcal{C}$ are uniquely determined by their right rules. If the condition of a generalized connective is merely $P\bot Q$, $Q\subseteq P^{\bot}$,  left rules of $\mathcal{C}$ is not uniquely determined. For example, we can consider different left rules $Q_1=\{(1,3)(2)\}$, $Q_2=\{(2,3)(1)\}$, $Q_3=\{\{(1,3)(2)\},\{(2,3)(1)\}\}$ for the right rule $P=\{(1,2)(3)\}$.


\begin{defn}
\cite{DR} A  generalized connective $\mathcal{C}$ is decomposable if $P_{\mathcal{C}}=P_{\alpha}$ for some $\mathsf{MLL}$-formula $\alpha$. Otherwise, $\mathcal{C}$ is said to be non-decomposable.
\end{defn}


Danos and Regnier gave the following example of non-decomposable connectives \cite{DR}.

\begin{center}
\def\fCenter{\ \vdash\ }
\Axiom$\fCenter\Gamma_1, A,B$
\Axiom$\fCenter\Gamma_2, C, D$
\BinaryInf$ \fCenter\Gamma_1, \Gamma_2, \mathcal{C}(A, B, C, D)$
\DisplayProof

\medskip

\def\fCenter{\ \vdash\ }
\Axiom$\fCenter\Gamma_1, A, C$
\Axiom$\fCenter\Gamma_2, B, D$ 
\BinaryInf$ \fCenter\Gamma_1,\Gamma_2, \mathcal{C}(A, B, C, D)$
\DisplayProof

\end{center}

\begin{center}
\def\fCenter{\ \vdash\ }
\Axiom$\fCenter\Delta_1, A, D$
\Axiom$\fCenter\Delta_2, B$
\Axiom$\fCenter\Delta_3, C$
\TrinaryInf$ \fCenter\Delta_1,\Delta_2,\Delta_3, \mathcal{C^{\ast}}(A, B, C, D)$
\DisplayProof

\medskip

\def\fCenter{\ \vdash\ }
\Axiom$\fCenter\Delta_1, B, C$
\Axiom$\fCenter\Delta_2, A$
\Axiom$\fCenter\Delta_3, D$
\TrinaryInf$ \fCenter\Delta_1,\Delta_2,\Delta_3,\mathcal{C^{\ast}}(A, B, C, D)$
\DisplayProof
\end{center}

Jean-Yves Girard gave other examples of non-decomposable  connectives in his paper ``transcendental syntax II: non deterministic case" in 2017.

Generalized connectives $\mathcal{C}, \mathcal{C}^{\ast}$ satisfy the main step of the cut-elimination  by the following Lemma.

\begin{lem}\label{DR lemma}
\cite[Lemma 1]{DR}   
For any $n$-ary connectives $\mathcal{C}_1$ and $\mathcal{C}_2$, the main cut-elimination step holds for any cut between any of the introduction rules for $\mathcal{C}_1$ and any of the introduction rules for $\mathcal{C}_2$ if and only if the corresponding partition[ sets] are orthogonal.
\end{lem}

We extend $\mathsf{MLL}$ by adding generalized connectives to $\mathsf{MLL}(\mathcal{C}_i)$ (where $\mathcal{C}_i   \ \ (i=1,\dots, n)$ is generalized connectives and each one has the arity $m_i$).  We identify the two binary generalized connectives with the tensor and par, respectively.  We also identify the $n$-ary tensor and par connectives with the corresponding generalized connectives. 

Formulas of $\mathsf{MLL}(\mathcal{C}_i)$   are defined as follows:

\medskip
$A:= P | \oldneg P | A\tensor A | A\pa A | \mathcal{C}_i(A, \dots, A)  \text{for each}\ i$
,where $P$ ranges over a denumerable set of propositional variables.
\medskip

The negation sign is used as the abbreviation, following Danos-Regnier's notation, although the negation symbol does not always represent logical negation because of the appearance of the generalized connectives;   
$\oldneg(\mathcal{C}_i(A_1,\dots,A_{m_i}))= \mathcal{C}^{\ast}_i(\oldneg A_1,\dots,\oldneg A_{m_i})$, $\oldneg(\mathcal{C}^{\ast}_i(A_1,\dots,A_{m_i}))= \mathcal{C}_i(\oldneg A_1,\dots,\oldneg A_{m_i})$. Observe that each generalized connective $\mathcal{C}_i$ determines its own negation by the cut-eliminationability. Here, for readability we use  the symbol ``$\oldneg$" to represent the orthogonal dual even though it does not mean the negation in the usual sense of MLL.
 
The set of inference rules of $\mathsf{MLL}(\mathcal{C}_i)$ is the union of that of $\mathsf{MLL}$ and that of $\mathcal{C}_i  (i=1,\dots,n)$.

\medskip

 We also extend the definition of links and proof-structures:  the node of label $\mathcal{C}_i$ has $m_i$ premises and one conclusion. If $A_1,\dots,A_{m_i}$ are the labels of the  premises, then the conclusion is labelled $\mathcal{C}_i(A_1,\dots,A_{m_i})$. We call this a $\mathbb{C}$-link. A link $l$ is a terminal link of a proof-structure containing  $\mathbb{C}$-links $\mathcal{S}$ when the conclusions of $l$ are conclusions of $\mathcal{S}$. 
 
  The cut-elimination theorem of $\mathsf{MLL}(\mathcal{C}_i)$  follows immediately from the above lemma \ref{DR lemma}.
\medskip

\begin{prop}\cite{DR}
Any provable sequent in $\mathsf{MLL}(\mathcal{C}_i)$ is provable without the cut rule.
\end{prop}

\section{Partition switching for the sequentialization theorem}

We introduce  a new definition of switching  for the generalized multiplicative connectives such that the switching condition implies the sequentialization theorem. We call our new switching the ``partition switching", because we think that our switching definition is a natural one from the viewpoint of the underlying partitions. We call Danos and Regnier's original switching for the generalized connectives ``Danos-Regnier switching", in this Section. The Danos-Regnier switching has the expansion property to be explained later. However, the sequentialization theorem for generalized connectives does not hold in general with their switching; Danos-Regnier  remark that,  in general, a sequent  $\vdash\mathcal{C}(A_1,\dots,A_n), \mathcal{C}^{\ast}(\oldneg A_1,\dots,\oldneg A_n)$ is not provable, and they added ``Hence, we cannot expect every correct proof-net to be sequentializable \cite[p.197]{DR}." %
Their view on the failure of  the sequentialization Theorem is based on their switching condition.  Thus, the fact that a sequent with excluded middle conclusion  is not provable does not by itself imply the failure of  the sequentialization Theorem. %
 We propose the partition switching, for generalized connectives, and using our new switching, we show that the sequentialization theorem holds for generalized connectives. However, as we explain later, the expansion property does not hold for our switching.

\bigskip

Our definition of  partition switching is as follows. 
\begin{defn}
Let $\mathcal{S}$ be a proof-structure containing $\mathbb{C}$-links $\mathcal{C}_i$ $(i=1,\dots,n)$. \newline A  partition switching $I$ of $\mathcal{S}$ is a family of functions indexed by partitions, \newline $f_{p}: P\restriction{p}\rightarrow \mathcal{P}(\mathcal{P}(\mathbb{N}))$ (where $P$ is a partition set and $p\in P$) such that $f_{p}(p)=\{p_{1f_{p}(1)},\dots,p_{kf_{p}(k)}\}$  (where $p=\{(p_{11},\dots,p_{1m_1}), \dots,(p_{k1},\dots,p_{km_k})\} , p_{jl}\in\{1,\dots \newline,  n \},  p_{if_{p}(i)}\in\{p_{i1},\dots , p_{im_i}\}$).
\end{defn}


\begin{defn}
Let $\mathcal{S}$ be a proof-structure containing $\mathbb{C}$-links $\mathcal{C}_i$ $(i=1,\dots,n)$ and $I=\{f_1,\dots, f_r\}$ be an arbitrary partition switching. $\mathcal{S}_{f_{p}}$ (where $f_p\in I$) is obtained from $\mathcal{S}$ by deleting all edges from $A_j$ $(j=1,\dots,n)$ to $\mathcal{C}_i$-node except those corresponding to $f_{p}$. We call  $\mathcal{S}_{f_{p}}$ the correctness graph of $\mathcal{S}$ for each $f_p$. We define $\mathcal{S}_{I}$ as follows: $\mathcal{S}_{I}=\bigcup_{f_p\in I}\mathcal{S}_{f_{p}}$.
\end{defn}

\begin{defn}
Let $\mathcal{S}$ be a proof-structure containing $\mathbb{C}$-links $\mathcal{C}_i$ $(i=1,\dots,n)$ and $I=\{f_1,\dots, f_r\}$ be an arbitrary partition switching. $\mathcal{S}$ is correct in the sense of partition switching if and only if there exists $f_{p}\in I$ such that  $\mathcal{S}_{f_{p}}$ is connected and acyclic. A proof net in the sense of partition switching is a correct proof-structure in the sense of partition switching.

\end{defn}

We give an example of  partition switching:    $\mathcal{C}(A_1,A_2,A_3)=(A_1\tensor A_2)\pa A_3$, $P=\{p_1, p_2\}$. For $p_1$, we choose one element from each class (e.g. $a_1, a_2$) and we obtain a switching $I_1=\{a_1, a_2\}$. The case of $p_2$ is similar.
\begin{center}
 $f_{p_1}: p_1=\{(a_1,  a_3),(a_2)\}\Longrightarrow_{I_1} a_1, a_2 $

 $f_{p_1}: p_1=\{(a_1,  a_3),(a_2)\}\Longrightarrow_{I_2} a_3, a_2 $
\medskip

or

 $f_{p_2}: p_2=\{(a_1),(a_2, a_3)\}\Longrightarrow_{J_1} a_1, a_2 $

$f_{p_2}: p_2=\{(a_1),(a_2, a_3)\}\Longrightarrow_{J_2} a_1, a_3 $

\end{center}

See Fig. \ref{partition switching graph} for the corresponding graphs of the above switchings.

\begin{figure}[h]
\begin{framed}
\begin{center}
\begin{tikzpicture}
\node at (0,0){$A_1$};
\node at (1,0) {$A_2$}     ;
\node at (2,0){$A_3$};
\node at (1,-1.2){$\mathcal{C}(A_1,A_2,A_3)$};
\draw[thick](0.1,-0.2)--(1,-0.8);
\draw[thick](1.1,-0.2)--(1.1,-0.8);
\end{tikzpicture}
\medskip
\begin{tikzpicture}
\node at (0,0){$A_1$};
\node at (1,0) {$A_2$}     ;
\node at (2,0){$A_3$};
\node at (1,-1.2){$\mathcal{C}(A_1,A_2,A_3)$};
\draw[thick](1.9,-0.2)--(1.2,-0.8);
\draw[thick](1.1,-0.2)--(1.1,-0.8);

\end{tikzpicture}
or
\begin{tikzpicture}
\node at (0,0){$A_1$};
\node at (1,0) {$A_2$}     ;
\node at (2,0){$A_3$};
\node at (1,-1.2){$\mathcal{C}(A_1,A_2,A_3)$};

\draw[thick](0.1,-0.2)--(1,-0.8);
\draw[thick](1.1,-0.2)--(1.1,-0.8);
\end{tikzpicture}
\medskip
\begin{tikzpicture}
\node at (0,0){$A_1$};
\node at (1,0) {$A_2$}     ;
\node at (2,0){$A_3$};
\node at (1,-1.2){$\mathcal{C}(A_1,A_2,A_3)$};
\draw[thick](1.9,-0.2)--(1.2,-0.8);
\draw[thick](0.1,-0.2)--(1,-0.8);
\end{tikzpicture}
\end{center}
\caption{Example of the partition switching}\label{partition switching graph}
\end{framed}
\end{figure}

%

\begin{rem}
 The partition switching is a generalization of $\pa$-switching because if we restrict the partition switching  to the binary case, it coincides with the usual $\pa$-switching.  In fact, the switching for generalized connectives satisfies the following conditions when it is restricted to the binary case; (i) a switching chooses at least one element from each class because the $\tensor$ connective has no partition switching and $P_{\tensor}=\{(1),(2)\}$ holds; and (ii) a switching  chooses at most one element from each class because the $\pa$-connective has  partition switching and $P_{\pa}=\{(1, 2)\}$ holds.

\end{rem}

We now show that by replacing the Danos-Regnier switching by our partition switching a  proof of the sequentialization theorem works.
\medskip


\begin{thm}\label{gen. seq}(Sequentialization Theorem for generalized multiplicative connectives)
Let $\mathcal{S}$ be an arbitrary proof-structure containing arbitrary $\mathbb{C}$-links $\mathcal{C}_i$   $(i=1,\dots, n)$. If $\mathcal{S}$ is a proof-net in the sense of the partition switching, then $\mathcal{S}$ is sequentializable.
\end{thm}

We shall give a slightly modified version of the proof for binary connectives given by Olivier Laurent in his unpublished note ``Sequentialization of multiplicative proof nets" at 2013 (available at: http://perso.ens-lyon.fr/olivier.laurent \newline /seqmill.pdf) [Accessed 1 April 2018] (Laurent, 2013). Our proof of Sequentialization Theorem for generalized multiplicative connectives contains a proof of Sequentialization Theorem for the binary multiplicative connectives as the special case because we can treat the binary connectives, par and tensor, as the special case (the case $n=2$) of generalized $n$-ary connectives, as remarked above. We often denote $n$-ary $\pa$ as  $\pa^n$.


\bigskip
\begin{sublem}[cf. (Laurent, 2013, Lemma 2)]\label{seq thm sublemma}
 Let $\mathcal{S}$ be a proof-net which does not contain any terminal $\pa^k$-links (for any non-zero natural number $k$) and which contains $n>1$ terminal $\mathbb{C}$-links  $\mathcal{C}_i$ ($i=1,\dots,n$). Here and in the next proof, a terminal $\tensor$-link is also regarded as a  terminal $\mathbb{C}$-link. If for an arbitrary terminal $\mathbb{C}$-link  $\mathcal{C}_i(A_{i1},\dots,A_{im})$,  $\mathcal{C}_i(A_{i1},\dots,A_{im})$ is non-splitting, then there exist a non-terminal node $l$ such that $l\neq\mathcal{C}_i(A_{i1},\dots,A_{im})$  and some paths $p_1, \dots, p_m$ such that each $ p_j$ starts from $\mathcal{C}_i(A_{i1},\dots,A_{im})$ and pass through $A_{i1},\dots,A_{im}$, respectively, and  meet at $l$.　
\end{sublem}
\begin{proof}
Let $\mathcal{C}_1(A_{11},\dots,A_{1m})$ be one of non-splitting terminal $\mathbb{C}$-link  and $\mathcal{C}_2(A_{21},\dots \newline,A_{2m})$ be an arbitrary distinct terminal  $\mathbb{C}$-link. We obtain some path from $\mathcal{C}_1$ to $\mathcal{C}_2$ by any switching because $\mathcal{S}$ is correct \footnote{We consider a proof-structure as if it is a directed graph in this proof.}.
We assume that a node $l$ such that $l\neq\mathcal{C}_i(A_{i1},\dots,A_{im})$ at which some paths $p_1, \dots, p_m$ meet does not exist, then we show a contradiction. There are upwards paths $p_1, \dots, p_m$ that start from $\mathcal{C}_1(A_{11},\dots,A_{1m})$ and pass through $A_{11},\dots,A_{1m_1}$, respectively. Then, at least two paths $p_s, p_t$ ($s, t=1,\dots, m$) has distinct targets $\alpha, \beta$ by assumption and there is no path from $\alpha$ to $\beta$ that does not pass through $\mathcal{C}_1(A_{11},\dots,A_{1m})$. Hence, if we remove $\mathcal{C}_1(A_{11},\dots,A_{1m})$, $\mathcal{S}$ is separated into two parts, the part containing $\alpha$ and the other containing $\beta$. It contradicts the assumption that $\mathcal{C}_1$ is non-splitting. Moreover, if $l$ is a terminal $\mathbb{C}$-node, then  $\mathcal{S}$ contains $n+1$ terminal $\mathbb{C}$-links and it contradicts the assumption. Hence $l$ is non-terminal.


\end{proof}

\begin{lem}[cf. (Laurent, 2013, p.4)]\label{Splitting}(Splitting Lemma)
Let $\mathcal{S}$ be a proof-net that does not contain terminal $\pa^{k}$-links and contains $n>0$ terminal $\mathbb{C}$-links $\mathcal{C}_i$ $(i=1,\dots,n)$. Then, at least one of terminal $\mathbb{C}$-links split.

\end{lem}

\begin{proof}
Case $n=1$;

We assume that the terminal $\mathbb{C}$-links $\mathcal{C}_1(A_1,\dots,A_m)$ is non-splitting. By Sublemma \ref{seq thm sublemma}, there exists a non-terminal node $l$ such that $l\neq\mathcal{C}_1(A_1,\dots,A_m)$. We obtain a upward path $p$  from $\mathcal{C}_1(A_1,\dots,A_m)$ to $l$ in $\mathcal{S}_I$, where $I$ is an arbitrary switching. There is a downward path $q$ such that $q\neq p$ from $l$ to $\mathcal{C}_1(A_1,\dots,A_m)$ in $\mathcal{S}_I$ because $l$ is non-terminal and $\mathcal{C}_1(A_1,\dots,A_m)$ is the only terminal node. We obtain a cycle $p\cdot q$ (where ``$\cdot$" denotes a concatenation of two paths), which contradicts the assumption that $\mathcal{S}$ is a proof-net. 


Case $n>1$.　

The Figure \ref{proof splitting lemma} will be helpful to fo follow our proof.  We assume that all  the terminal $\mathbb{C}$-links $\mathcal{C}_i$'s are non-splitting and we show a contradiction. By Sublemma, for each $i$, there are some non-terminal node $l_i$ and some upwards paths $p_{i1}, \dots, p_{im}$ such that $p_{i1}, \dots, p_{im}$ meet at $l_i$,  paths $p_{i1}, \dots, p_{im}$ start from $\mathcal{C}_i(A_{i1},\dots, A_{im})$ and pass through $A_{i1},\dots,A_{im}$, respectively. For all $j$ ($j=1,\dots, n$), there are upward paths $q_j$ from $\mathcal{C}_j$ to $l_{j}$ in $\mathcal{S}_I$ (where $I$ is an arbitrary switching) by the assumption that $\mathcal{S}$ does not contain terminal $\pa^{k}$-links and $\mathcal{S}$ is correct.  We can obtain downward paths $d_j$ in $\mathcal{S}_I$ (where $I$ is an arbitrary switching) which start from $l_j$ and end with the terminal node because $l_j$ are non-terminal. There exists a path $q$ such that it starts at $l_s$  for some $s$ ($1\leq s\leq n$) and ends with $\mathcal{C}_t$ for some $t$, where $t<s$ (say, $\mathcal{C}_1(A_{11},\dots,A_{1m})$). This is  because $l_i$ ($i=1,\dots, n$) are non-terminal and  $\mathcal{S}$  contains only $n>0$ terminal links. 
If for an arbitrary switching $J$, the path $q$ is  contained in $\mathcal{S}_J$, we can construct  the path $q_1\cdot d_1\cdots\cdot q_s\cdot q$ . But, it is a cycle, which contradicts the correctness of $\mathcal{S}$. If for any switching $J$, $\mathcal{S}_J$ does not contain the path $u$, then  $u$ is separated by two parts $u_1$ and $u_2$ at some $\mathbb{C}$-node $l'$, then $l'$ is non-terminal by the assumption. It follows that there is a path $r$ from $l'$ to $\mathcal{C}_t(A_{t1},\dots,A_{tm_t})$ for some $t$ such that  $t<s$. We put $u'=u_1\cdot r$. Without loss of generality, we assume that $\mathcal{S}_J$ contains the path $u'$. We can construct the path $q_1\cdot d_1\cdots q_s\cdot u'$. It is a cycle, which contradicts the correctness of $\mathcal{S}$.
\end{proof}

We return to the proof of  Theorem \ref{gen. seq}.

\begin{proof}
(Proof of Sequentialization Theorem)  \\　By induction on the number of links contained in $\mathcal{S}$. We  prove it for the case that the terminal link is $\mathbb{C}$-link. If necessary, first we remove all terminal $\pa^{k}$-links from $\mathcal{S}$. By Lemma \ref{Splitting}, at least one of the terminal $\mathbb{C}$-links $\mathcal{C}_1(A_{1},\dots,A_{m})$ is splitting. Hence the removal of $\mathcal{C}_1(A_{1},\dots,A_{m})$ splits $\mathcal{S}$ into several sub-proof-nets $\{S_1,\dots,S_i\}$. By the induction hypothesis, there are some proofs $\pi_1,\dots,\pi_i$ such that $(\pi_i)^{\ast}=\mathcal{S}_i$ holds.  We apply the $\mathcal{C}$-rule to $\pi_1,\dots,\pi_i$ and obtain the proof $\pi$.

\end{proof}

We showed above that Olivier Laurent's direct proof of the sequentialization theorem for the binary connectives can be adapted to the $n$-ary generalized connectives when we take our partition-switching as a generalization of the binary par-switching, as the above proof essentially follows Laurent's proof. At the same time, we presented a slightly simplified proof for the splitting lemma. 


Figure \ref{proof splitting lemma} represents the $n>1$ case of the proof above.
\medskip

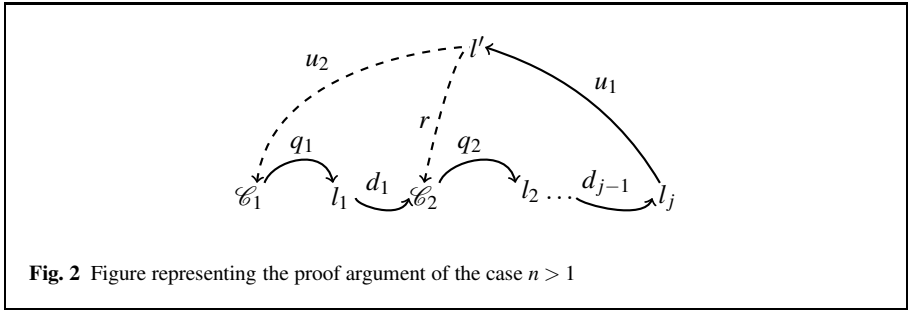
\begin{figure}[h]
\begin{framed}
\begin{center}
\begin{tikzpicture}
\node at (-0.2,-0.2){$\mathcal{C}_1$};
\draw[thick, ->](0,0)..controls(0.1,0.3)and(0.8,0.5)..(0.9,0);
\node at (0.5,0.5){$q_1$};
\node at (1,-0.2){$l_1$};
\draw[thick, ->](1.2,-0.2)..controls(1.2,-0.3)and(1.8,-0.5)..(1.9,-0.2);
\node at (1.5,0){$d_1$};
\node at (2.1,-0.2){$\mathcal{C}_2$};
\draw[thick, ->](2.3,0)..controls(2.4,0.3)and(3.2,0.5)..(3.3,0);
\node at (2.7,0.5){$q_2$};
\node at (3.5,-0.1){$l_2$};
\node at (3.9,-0.2){$\dots$};
\draw[thick, ->](4.1,-0.2)..controls(4.2,-0.3)and(5,-0.5)..(5.1,-0.2);
\node at (4.5,0){$d_{j-1}$};
\node at (5.3,-0.2){$l_j$};
\draw[thick, <-, dashed](-0.1,0)..controls(0.1,1.6)and(2.5,1.8)..(2.6,1.8);
\node at (4.5,1.3){$u_1$};
\node at (2.8,1.8){$l'$};
\draw[thick, <-](2.9,1.8)..controls(3.5, 1.6)and(4.5,1.2)..(5.2,0);
\draw[thick, <-, dashed](2.1,0)..controls(2.2,0.5)and(2.5,1.8)..(2.7,1.8);
\node at (2.1,0.8){$r$};
\node at (0.7,1.6){$u_2$};
\end{tikzpicture}
\end{center}
\caption{Figure representing the proof argument of the case $n>1$}\label{proof splitting lemma}
\end{framed}
\end{figure}

The converse direction of Sequentialization Theorem also holds with our partition switching as it is in the original binary case.

\begin{prop}\label{gen. seq2}
Let $\mathcal{S}$ be an arbitrary proof-structure containing arbitrary $\mathbb{C}$-links $\mathcal{C}_i$   $(i=1,\dots, n)$. If $\mathcal{S}$ is sequentializable, then $\mathcal{S}$ is a proof-net  in the sense of  the partition switching.
\end{prop}

\begin{proof}
By induction on the length of a proof.  It follows from the definition of the partition switching. 
\end{proof}

Next, we explain the Danos and Regnier switching and compare their switching with ours.

The definition of the Danos-Regnier switching is as follows. (We regard a class of a partition as a set of elements $p_{ij}$).
\begin{defn}
(Danos-Regnier switching\cite{DR}) Let $\mathcal{S}$ be a proof-structure containing $\mathbb{C}$-links $\mathcal{C}_i$ $(i=1,\dots,n)$.
For an arbitrary $i\in\{1,\dots, n \}$ and $p\in P_{\mathcal{C}}$ (where  $p=\{(p_{11},\dots,p_{1j}),\dots,(p_{k1},\dots,p_{km_k})\}$, $p_{jk}\in\{1,\dots, n \}$), a Danos-Regnier switching $I$ is a  function $f: P\rightarrow Class(P)$ (where $P$ is a partition set)  such that  $f(p)=x_i\}$ ($x_i\in Class(p)$ and $\{p_{i1},\dots, p_{im_i}\}=x_i$).
\end{defn}

\begin{defn}(Correctness graph for Danos-Regnier switching)
Let $\mathcal{S}$ be a proof-structure containing $\mathbb{C}$-links $\mathcal{C}_i$ $(i=1,\dots,n)$ and $I$ be an arbitrary Danos-Regnier switching. The correctness graph $\mathcal{S}_I$ is obtained from  $\mathcal{S}$ and $I$ as follows; for each class, connect the nodes that belong to the same class and cut off the edges from classes to the $\mathcal{C}_i$-node except the class $x=f(p)$ i.e. the value of $I$.

\end{defn}

A proof-structures that has two nodes $\mathcal{C}(A_1,\dots,A_n), \mathcal{C}^{\ast}(\oldneg A_1,\dots,\oldneg A_n)$ for some atoms $A_1,\dots,A_n$  is counterexample of the sequentialization theorem using Danos-Regnier  switchings. It is easily shown as follows: by Fact \ref{expansion 1}, a proof-structure  $\mathcal{S}$  that has just two terminal nodes $\mathcal{C}(A_1,\dots,A_n)$ and $\mathcal{C}^{\ast}(\oldneg A_1,\dots,\oldneg A_n)$ for some $A_1,\dots,A_n$ is correct. However, the corresponding sequent $\vdash\mathcal{C}(A_1,\dots, A_n)$, $\mathcal{C^{\ast}}(\oldneg A_1,\dots, \oldneg A_n)$ does not necessarily have proof by Fact \ref{axiom deducibility}.

\begin{fact1}\label{expansion 1} 
If $\mathcal{S}$ is a proof-structure  that has just two terminal nodes $\mathcal{C}(A_1,\dots,A_n)$ and $\mathcal{C}^{\ast}(\oldneg A_1,\dots,\oldneg A_n)$ for some $A_1,\dots,A_n$,  then $\mathcal{S}$ is correct in the sense of Danos-Regnier switching.
\end{fact1}

We call the  terminal nodes with  labels  ``$\mathcal{C}, \mathcal{C}^{\ast}$" the excluded middle formula (with respect to $\mathcal{C}$).

\begin{fact1}\label{axiom deducibility}
A sequent $\vdash\mathcal{C}(A_1,\dots, A_n)$, $\mathcal{C^{\ast}}(\oldneg A_1,\dots, \oldneg A_n)$ for any dual pair of generalized connectives is provable if and only if $\mathcal{C}$=$\tensor^k$ or $\mathcal{C}$=$\pa^k$.
\end{fact1}

We note that the Facts \ref{expansion 1} and \ref{axiom deducibility} are essentially (implicitly) in Danos-Regnier \cite{DR}.

\begin{rem}
The Danos-Regnier switching is a generalization of $\pa$-switching. By considering binary case, we obtain $I^1_{\tensor}=\{1\}$, $I^2_{\tensor}=\{2\}$, and $I_{\pa}=\{1,2\}$. The roles of $\tensor$ and $\pa$ on graphs are reversed; $\tensor$ has the binary switching and $\pa$ does not have it.  Hence, if we use the left-one-sided sequent $\mathsf{MLL}$, the sequentialization theorem hold for binary connectives. \end{rem}

As we remarked in Introduction (Section 1), the Danos-Regnier switching chooses one class from a partition, while our partition switching chooses exactly one element from each class of a given partition.

The following graph (Fig. \ref{ex. DR switching}) is an example of a graph $\mathcal{S}_I$ (where $\mathcal{S}$ is correct in the sense of Danos-Regnier switching and $\mathcal{C}$ and $\mathcal{C}^{\ast}$ are non-decomposable connectives of Danos-Regnier). If we delete $\mathcal{C}$ and $\mathcal{C}^{\ast}$-nodes and its vertical edges, then we obtain the meeting graph $\mathcal{G}(P_{\mathcal{C}}, P_{\mathcal{C}^{\ast}})$. Conversely, if the meeting graph $\mathcal{G}(P_{\mathcal{C}}, P_{\mathcal{C}^{\ast}})$ is given, then we can choose two classes and add nodes  $\mathcal{C}$, $\mathcal{C}^{\ast}$ and two edges to chosen nodes. By this operation, we can obtain a correction  graph. How to connect $\mathbb{C}$-nodes with classes is not important in the definition of Danos-Regnier switching because it is irrelevant to correctness. Hence, a meeting graph is almost the same as a correction graph in the sense of  Danos-Regnier switching. 

\begin{figure}[h]
\begin{framed}

\begin{center}
\begin{tikzpicture}[scale=0.7]

\node at (0,0){$A_1$};
\node at (1,0) {$A_2$ }    ;
\node at (2,0){$A_3$};
\node at (3,0){$A_4$};
\node at (1,-1.5){$\mathcal{C}(A_1,A_2,A_3,A_4)$};
\draw[thick](0.1,-0.2)--(1,-0.8);
\draw[thick](2.1,-0.2)--(1,-0.8);

\draw[thick](1.1,-0.2)--(2,-0.8);
\draw[thick](3.1,-0.2)--(2,-0.8);

\draw[thick](1,-0.8)--(1,-1.3);
\draw[thick](5.6,-0.8)--(5.5,-1.2);

\node at (4,0){$\oldneg A_1$};
\node at (5,0) {$\oldneg A_2$}     ;
\node at (6,0){$\oldneg A_3$};
\node at (7,0){$\oldneg A_4$};
\node at (7,-1.5){$\mathcal{C}^{\ast}(\oldneg A_1, \oldneg A_2, \oldneg A_3, \oldneg A_4)$};

\draw[thick](4.1,-0.2)--(5.6,-0.8); 
\draw[thick](7,-0.2)--(5.6,-0.8);

\draw[thick, rounded corners=8pt](0,0.2)--(0,0.5)--(4,0.5)--(4,0.2);
\draw[thick, rounded corners=8pt](1,0.2)--(1,0.7)--(5,0.7)--(5,0.2);
\draw[thick, rounded corners=8pt](2,0.2)--(2,0.9)--(6,0.9)--(6,0.2);
\draw[thick, rounded corners=8pt](3,0.2)--(3,1.1)--(7,1.1)--(7,0.2);


\end{tikzpicture}
\end{center}
\caption{Example of Danos-Regnier switching}\label{ex. DR switching}
\end{framed}
\end{figure}

Axiom links are only allowed for atomic formulas in Definition \ref{links}. In the following part, we include axiom links for arbitrary formulas $A, \oldneg A$ in the definition of a link (and a proof-structure). %
We employ non-atomic initial axiom sequents and non-atomic initial axiom links so as to compare our switching condition and Danos and Regnier's one.  When we admit axiom links and non-atomic initial sequents for any formulas, a trade-off relation which will be explained later arises.
\medskip

We extend the axiom-rule and the axiom link as follows;

\begin{center}
\AxiomC{}
\RightLabel{(Ax)}
\UnaryInfC{$\ \vdash\  A, \oldneg A$}
\DisplayProof
\medskip
\begin{tikzpicture}
\node at (0,0){$A$};
\node at (2,0){$\oldneg A$}    ;
\draw[thick, rounded corners=8pt](0,0.2)--(0,0.5)--(2,0.5)--(2,0.2); 
\node at (3.5,0){(Axiom)};
\end{tikzpicture}
(where $A$ is an arbitrary formula which may include the generalized connectives.)
\end{center}

\medskip
\par We can show the sequentialization theorem for a proof-structure $\mathcal{S}$ containing non-atomic axioms and generalized connectives by the same proof as Theorem \ref{gen. seq}.

\begin{defn}
Let $\mathcal{S}$ be a proof-structure containing $\mathbb{C}$-links $\mathcal{C}_i$ $(i=1,\dots,n)$ and  $\mathbb{C}$-axiom links $L_j  (1\leq j\leq \frac{n}{2} $), (where $j$ is a natural number). The axiom expansion is the following operation $F$; $F(\mathcal{S})=\mathcal{S}'$ where $\mathcal{S}'$ is the same as $\mathcal{S}$ except that each  $\mathbb{C}$-axiom links $L_j $ is replaced with its subformula axiom links (For example, Fig. \ref{axiom expansion}).

\begin{figure}[h]
\begin{framed}

\begin{center}
\begin{tikzpicture}
\node at (0,0){$\mathcal{C}(A,\dots,D)$};
\node at (3,0){$\mathcal{C}^{\ast}(\oldneg A\dots,\oldneg D)$}     ;
\draw[thick, rounded corners=8pt](0,0.2)--(0,0.5)--(3,0.5)--(3,0.2);
\end{tikzpicture}
\medskip
$\leadsto$

\begin{tikzpicture}

\node at (0,0){$A$};
\node at (1,0) {$B$}     ;
\node at (2,0){$C$};
\node at (3,0){$D$};
\node at (1,-1.5){$\mathcal{C}(A, B, C, D)$};
\draw[thick](0.1,-0.2)--(0.8,-1.1);
\draw[thick](1.1,-0.2)--(1,-1.1);
\draw[thick](2.1,-0.2)--(1.1,-1.1);
\draw[thick](3.1,-0.2)--(1.2,-1.1);

\node at (4,0){$\oldneg A$};
\node at (5,0) {$\oldneg B$}     ;
\node at (6,0){$\oldneg C$};
\node at (7,0){$\oldneg D$};
\node at (5.5,-1.5){$\mathcal{C}^{\ast}(\oldneg A, \oldneg B, \oldneg C, \oldneg D)$};

\draw[thick](4.1,-0.2)--(5.3,-1.1); 
\draw[thick](5.1,-0.2)--(5.45,-1.1); 
\draw[thick](6,-0.2)--(5.6,-1.1);
\draw[thick](7,-0.2)--(5.8,-1.1);

\draw[thick, rounded corners=8pt](0,0.2)--(0,0.5)--(4,0.5)--(4,0.2);
\draw[thick, rounded corners=8pt](1,0.2)--(1,0.7)--(5,0.7)--(5,0.2);
\draw[thick, rounded corners=8pt](2,0.2)--(2,0.9)--(6,0.9)--(6,0.2);
\draw[thick, rounded corners=8pt](3,0.2)--(3,1.1)--(7,1.1)--(7,0.2);


\end{tikzpicture}
\end{center}

\caption{Axiom expansion : the upper diagram is expanded to the lower diagram. }\label{axiom expansion}
\end{framed}
\end{figure}

 \end{defn}

\begin{prop}\label{expansion DR}
Let $\mathcal{S}$ be a proof-net containing $\mathbb{C}$-links $\mathcal{C}_i$ $(i=1,\dots,n)$ and  $\mathbb{C}$-axiom links $L_j  (1\leq j\leq\frac{n}{2} $).  the proof-structure $\mathcal{S}'$  expanded from $\mathcal{S}$ is correct in the sense of  Danos-Regnier switching.
\end{prop} 

\begin{proof}
For any $\mathbb{C}$-axiom link $L_j$ (say, $L_1$), $\mathcal{C}_1$-node is connected to the connected and acyclic subgraph $\mathcal{S}_1$.  Similarly, $\mathcal{C}^{\ast}_1$-node is connected to the correct subgraph $\mathcal{S}_2$. $\mathcal{S}'$ consists of $\mathcal{S}_1$, $\mathcal{S}_2$ and the proof-structure containing two terminal links $\mathcal{C}_1$ and $\mathcal{C}_2$. By Fact \ref{expansion 1}, $\mathcal{S}'$ is correct.

\end{proof}

 Their switching condition guarantees the expansion property telling that the proof-structure of any excluded middle formula satisfies their switching condition (Fact \ref{expansion 1}). The expansion property implies that any non-atomic axiom link can be expanded to the proof structure from atomic axiom links satisfying the switching condition (Proposition \ref{expansion DR}). This provides a certain connection of the dual connective $\mathcal{C}^{\ast}$ and logical negation through the atomic negations.

\bigskip

We now introduce the notion of $n$-ary $\pa$-switching ($n$ is an arbitrarily fixed non-zero natural number).

\begin{defn}
Let $\mathcal{S}$ be an arbitrary proof-structure containing  $\pa_j,  j\in\{1,\dots, i\}$ ($i$ is a non-zero natural number). For an arbitrary $j\in\{1,\dots, n\}$, we write the arity of $\pa_j$ as $a(j)$. We denote the set of all n-ary $\pa$-links contained in a proof-structure $\mathcal{S}$ by $\pa(\mathcal{S})$. A $n$-ary $\pa$-switching $I$ of a proof-structure $\mathcal{S}$ is a function $f: \pa_j\in\pa(\mathcal{S})\to\{1,\dots,a(j)\}$. The graph $\mathcal{S}_I$ is obtained by deleting the edges of each n-ary par link $\pa_j$ except the number of $I$.

\end{defn}

\begin{fact1}\label{par^n coincidence}
Let  $\mathcal{C}$ be a generalized connective such that $\mathcal{C}=\pa^n$ holds. Then, the partition switching and the $\pa^n$-switching  coincide.
\end{fact1}

This $n$-ary  $\pa$-switching is, of course, decomposable and essentially  reduced to the binary connective $\pa$.

\medskip

We summarize the trade-off relationship between the sequentialization theorem and the expansion property;  a proof-net in the sense of Danos-Regnier switching is not necessarily sequentializable, as we explained after Definition 16, while any proof-net in the sense of our partition switching is sequentializable. On the other hand, the expansion property does not hold for our partition switching, while it holds for Danos-Regnier's one.  Note that these trade-off relationship occur only when a proof-structure contains non-atomic axiom links. We formulate the trade-off relationship more precisely. The next proposition says that the expansion property and the sequentialization theorem are not compatible in general. 

\begin{prop}\label{correct-seq}
 A switching $f:p\to X$ satisfies both the expansion property and the sequentialization theorem if and only if $f$ is the n-ary $\pa$-switching.
\end{prop}
\begin{proof}
 (only-if part)  Let $f$ be an arbitrary switching for generalized connectives that satisfies both the excluded middle property and the sequentialization theorem. We assume that $f$ is not $n$-ary $\pa$-switching and we show a contradiction. By the expansion property, the proof-structure $\mathcal{S}$ containing two terminal nodes $\mathcal{C}(A,B,C)=(A\pa B)\tensor C$, $\mathcal{C}^{\ast}(\oldneg A,\oldneg B,\oldneg C)=(\oldneg A\tensor\oldneg B)\pa\oldneg C$ (where $A, B$ and $C$ are atoms) is correct. By the sequentialization theorem, we obtain a proof $\pi$ of the sequent $\vdash \mathcal{C}(A,B,C), \mathcal{C}^{\ast}(\oldneg A,\oldneg B,\oldneg C)$ from $\mathcal{S}$.
This contradicts Fact \ref{axiom deducibility}.

(if-part) Let $I$ be the $n$-ary par switching. The proof-structure $\mathcal{S}$ containing exactly two terminal nodes $\pa^n$ and $\tensor^{n}$ is correct. Hence, the expansion property holds. By Fact \ref{par^n coincidence} and Theorem \ref{gen. seq}, Sequentialization Theorem follows.

\end{proof}

\section{Conclusions}

Danos and Regnier's switching condition implies the expansion property but does not imply the sequentialization property.  We have introduced a new switching condition which implies the sequentialization property but  does not imply the expansion property. Thus, except in the purely tensor-based or purely par-based cases, sequentialization and expansion cannot be achieved by a single  $n$-ary switching notion.

\section*{Acknowledgement}
We are grateful to an anonymous reviewer for helpful comments on the earlier version of this paper. This work was supported by MEXT-Grant-in-Aid for Scientific Research (B) (JP17H02265), JSPS- AYAME Program and Promotion Program for Next Generation Research Projects (2013-2017) of Keio University. The first author was also supported by Keio University doctorate student Grant-in-Aid Program.






\section{Appendix}

Multiplicative Linear Logic $\mathsf{MLL}$ \cite{Gir} is defined as follows.

\begin{defn}
The formulas of $\mathsf{MLL}$  are defined by

$$A::= P | \oldneg P | A\tensor A | A\pa A$$
where $P$ ranges over a denumerable set of propositional variables.
\medskip

Negation is defined as the abbreviation in the following way. Here, we sometimes call $\oldneg A$ the dual formula of $A$:
\medskip

$\oldneg\oldneg P:=P, \oldneg(A\tensor B):=\oldneg A\pa\oldneg B, \oldneg(A\pa B):=\oldneg A\tensor\oldneg B$
, where $P$ is an atomic formula.
\end{defn}

The inference rules of the sequent calculus $\mathsf{MLL}$ are given in Fig.~\ref{MLL rules}. We regard the finite sequence of formulas in a sequent as a multiset rather than an ordered sequence; therefore, the exchange rule is not needed.

  \begin{figure}[h]
\begin{framed}

\begin{center}
\AxiomC{}
\RightLabel{(Ax)}
\UnaryInfC{$\ \vdash\  P, \oldneg P$}
\DisplayProof
\medskip
\def\fCenter{\ \vdash\ }
\Axiom$\fCenter \Gamma,  A$
\Axiom$\fCenter\Delta, \oldneg A$
\RightLabel{(Cut)}
\BinaryInf$\fCenter\Gamma, \Delta$
\DisplayProof
\end{center}

\begin{center}
\def\fCenter{\ \vdash\ }
\Axiom$\fCenter \Gamma,  A$
\Axiom$\fCenter\Delta, B$
\RightLabel{(\tensor)}
\BinaryInf$\fCenter\Gamma, \Delta, A\tensor B$
\DisplayProof
\medskip
\Axiom$\fCenter \Gamma,  A, B$
\RightLabel{($\pa$)}
\UnaryInf$\fCenter\Gamma, A\pa B$
\DisplayProof
\end{center}
\begin{center}
 where $P$ is an atomic formula
\end{center}
\caption{ Inference rules of $\mathsf{MLL}$}\label{MLL rules}
\end{framed}
\end{figure}

\medskip

Now we introduce the basic notions of graphical representations of $\mathsf{MLL}$-proofs.

\medskip

\begin{defn}\label{links}
$\mathsf{MLL}$-links are the following four kinds of graphs (Fig. \ref{MLL-links}).

\begin{figure}[h]
\begin{framed}
\begin{center}
\begin{tikzpicture}
\node at (0,0){$P$};
\node at (2,0){$\oldneg P$}    ;
\draw[thick, rounded corners=8pt](0,0.2)--(0,0.5)--(2,0.5)--(2,0.2); 
\node at (3.3,0){(Axiom)};
\end{tikzpicture}
\bigskip
\begin{tikzpicture}
\node at (0,0){$A$};
\node at (2,0){$\oldneg A$}     ;
\node at (3,0){(Cut)};
\draw[thick, rounded corners=8pt](0,-0.2)--(0,-0.5)--(2,-0.5)--(2,-0.2);
\end{tikzpicture}
\end{center}

\begin{center}
\begin{tikzpicture}
\node at (0,0){$A$};
\node at (2,0) {$B$};
\node at (1,-1){$A\tensor B$};
\draw [thick](0.15,-0.1)--(0.9,-0.8);
\draw[thick](1.8,-0.1)--(1.1,-0.8);
\node at (3,-0.5){(Tensor)};
\end{tikzpicture}
\bigskip
\begin{tikzpicture}
\node at (0,0){$A$};
\node at (2,0) {$B$}     ;
\node at (1,-1){$A\pa B$};
\draw [thick](0.15,-0.1)--(0.9,-0.8);
\draw[thick](1.8,-0.1)--(1.1,-0.8);
\node at (3,-0.5){(Par)};
\end{tikzpicture}
\end{center}

 where $P$ is an atomic formula.

\caption{$\mathsf{MLL}$-links}\label{MLL-links}
\end{framed}
\end{figure}

\begin{itemize}
\item The node of each label $\tensor$ (resp. $\pa$) has two ordered premises and one conclusion. If $A$ is the label of the first premise and $B$ that of the second, then the conclusion is labelled $A\tensor B$ (resp. $A\pa B$);

\item The node of label ax has no premise and two ordered conclusions. If the label of the first conclusion is $P$, the label of the second conclusion is $\oldneg P$, where $P$ is an atomic formula;

\item The node of labelled cut has two ordered premises and no conclusion. If the label of the first conclusion is $A$, the label of the second conclusion is $\oldneg A$.

\end{itemize}
\end{defn}

We define the notion of proof-structure and of correct proof-structure or proof-net, as follows.
%
\begin{defn}
A proof-structure $\mathcal{S}$ is a non-empty finite undirected graph whose nodes are labelled by  $\mathsf{MLL}$-formulas and whose edges are labeled by the $\mathsf{MLL}$-links satisfying the following condition; each formula occurrence is the premise of at most one link and the conclusion of at most one link; terminal formulas are those that are conclusions of no link.

\end{defn}
  An instance of a formula in a proof-structure that is not the premise of any link is called terminal.


\begin{defn}
We denote the set of all par links contained in a proof-structure $\mathcal{S}$ by $\pa(\mathcal{S})$. A $\pa$-switching $I$ of a proof-structure $\mathcal{S}$ is a function $\pa(\mathcal{S})\to\{$left, right$\}$.
\end{defn}



\begin{defn}(Danos and Regnier\cite{DR})
A proof-structure $\mathcal{S}$ is correct if and only if for any switching $I$, the induced graph $\mathcal{S}_I$ is connected and acyclic. A proof-net $\mathcal{S}$ is a correct proof-structure. \end{defn}

The above condition is often called the switching condition. Then, one can say that a  proof-structure is a proof-net when it satisfies the switching condition.

\medskip

\begin{defn}

Let $\pi$ be a proof of $\mathsf{MLL}$ and $\mathcal{S(\pi)}$ be the proof-structure that is obtained from $\pi$ by applying the  translation $\mathcal{S(-)}$ written in Fig. \ref{translation MLL}  recursively.  

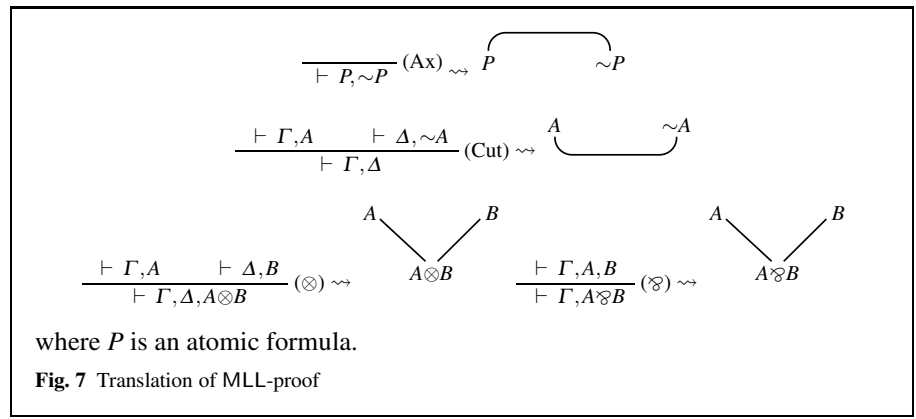
\begin{figure}[H]
\begin{framed}
\begin{center}
\scalebox{0.8}{
\AxiomC{}
\RightLabel{(Ax)}
\UnaryInfC{$\ \vdash\  P, \oldneg P$}
\DisplayProof
\medskip
$\leadsto$
\begin{tikzpicture}
\node at (0,0){$P$};
\node at (2,0){$\oldneg P$}     ;
\draw[thick, rounded corners=8pt](0,0.2)--(0,0.5)--(2,0.5)--(2,0.2);
\end{tikzpicture}}
\end{center}

\begin{center}
\scalebox{0.8}{
\def\fCenter{\ \vdash\ }
\Axiom$\fCenter \Gamma,  A$
\Axiom$\fCenter\Delta, \oldneg A$
\RightLabel{(Cut)}
\BinaryInf$\fCenter\Gamma, \Delta$
\DisplayProof
\medskip
$\leadsto$
\begin{tikzpicture}
\node at (0,0){$A$};
\node at (2,0){$\oldneg A$}     ;
\draw[thick, rounded corners=8pt](0,-0.2)--(0,-0.5)--(2,-0.5)--(2,-0.2);
\end{tikzpicture}}
\end{center}
\begin{center}
\scalebox{0.8}{
\def\fCenter{\ \vdash\ }
\Axiom$\fCenter \Gamma,  A$
\Axiom$\fCenter\Delta, B$
\RightLabel{(\tensor)}
\BinaryInf$\fCenter\Gamma, \Delta, A\tensor B$
\DisplayProof
\medskip
$\leadsto$
\begin{tikzpicture}
\node at (0,0){$A$};
\node at (2,0) {$B$}     ;
\node at (1,-1){$A\tensor B$};
\draw [thick](0.15,-0.1)--(0.9,-0.8);
\draw[thick](1.8,-0.1)--(1.1,-0.8);
\end{tikzpicture}

\def\fCenter{\ \vdash\ }
\Axiom$\fCenter \Gamma,  A, B$
\RightLabel{($\pa$)}
\UnaryInf$\fCenter\Gamma, A\pa B$
\DisplayProof
\medskip
$\leadsto$
\begin{tikzpicture}
\node at (0,0){$A$};
\node at (2,0) {$B$}     ;
\node at (1,-1){$A\pa B$};
\draw [thick](0.15,-0.1)--(0.9,-0.8);
\draw[thick](1.8,-0.1)--(1.1,-0.8);
\end{tikzpicture}}
\end{center}

where $P$ is an atomic formula.
\caption{Translation of  $\mathsf{MLL}$-proof}\label{translation MLL}
\end{framed}
\end{figure}

 A proof-structure $\mathcal{T}$  is said to be sequentializable if there exists a $\mathsf{MLL}$-proof $\pi$ such that $\mathcal{S(\pi)}=\mathcal{T}$ holds.

\end{defn}


\begin{prop}(Sequentialization of $\mathsf{MLL}$, Girard, Danos and Regnier) \cite{Gir, DR}

A proof-structure $\mathcal{S}$ is sequentializable if and only if $\mathcal{S}$ is a proof-net.
 
\end{prop}

Girard proved the sequentialization theorem using his long-trip condition \cite{Gir}. After that, Danos and Regnier gave the switching condition and stated above form of the theorem \cite{DR}.

\end{document}